\documentclass[11pt]{amsart}

\usepackage[a4paper,margin=1.1in]{geometry}
\usepackage[T1]{fontenc}
\usepackage{lmodern}
\usepackage{microtype}
\usepackage{amsmath,amssymb,mathtools}
\usepackage{enumitem}
\usepackage{xurl}
\usepackage{xcolor}
\usepackage{tikz}
\usetikzlibrary{arrows.meta,calc,positioning}
\usepackage[colorlinks=true,linkcolor=blue!45!black,citecolor=blue!45!black,urlcolor=blue!45!black]{hyperref}

\usepackage[all]{xy}
\usepackage{relsize}
\usepackage[bbgreekl]{mathbbol}

\DeclareSymbolFontAlphabet{\mathbb}{AMSb}
\DeclareSymbolFontAlphabet{\mathbbl}{bbold}

\numberwithin{equation}{section}

\theoremstyle{plain}
\newtheorem{thm}{Theorem}[section]
\newtheorem{prop}[thm]{Proposition}
\newtheorem{lem}[thm]{Lemma}
\newtheorem{cor}[thm]{Corollary}

\theoremstyle{definition}
\newtheorem{defn}[thm]{Definition}

\newtheorem{construction}[thm]{Construction}

\theoremstyle{remark}
\newtheorem{rem}[thm]{Remark}
\newtheorem{ex}[thm]{Example}

\newcommand{\Prism}{\mathbbl{\Delta}}
\newcommand{\crys}{\mathrm{crys}}
\newcommand{\AlgK}{\mathrm{K}}
\newcommand{\pd}{\mathrm{pd}}
\newcommand{\Vect}{\mathrm{Vect}}

\newcommand{\Z}{\mathbb{Z}}

\newcommand{\Q}{\mathbb{Q}}

\newcommand{\F}{\mathbb{F}}

\newcommand{\TC}{\mathrm{TC}}

\newcommand{\Shv}{\mathrm{Shv}}

\newcommand{\op}{\mathrm{op}}

\newcommand{\frakS}{\mathfrak{S}}

\newcommand{\calO}{\mathcal{O}}

\newcommand{\tilp}{\widetilde{p}}
\newcommand{\qdR}{\mathrm{qdR}}
\newcommand{\pdR}{\widetilde{p}\mathrm{dR}}

\newcommand{\rmD}{\mathrm{D}}
\newcommand{\hatD}{\widehat{\mathrm{D}}}

\DeclareMathOperator{\Spec}{Spec}

\title[A Note on Prismatic Sites for \texorpdfstring{$p$}{p}-Quasisyntomic Rings]
{A Note on Prismatic Sites for \texorpdfstring{$p$}{p}-Quasisyntomic Rings}

\begin{document}

\author{Jingbang Guo} 
\thanks{The author would thank the support from Guozhen Wang and the funding NSFC-12226002.}

\begin{abstract}
    Let $p$ be a fixed prime number and let $R$ be a $p$-quasisyntomic ring. In this note, we provide conditions for objects in the absolute prismatic site $R_\Prism$ to cover the final object in $\Shv(R_\Prism)$ and to admit finite self-coproducts in $R_\Prism$. More precisely, we introduce in $R_\Prism$ the so-called transversal objects, so that the coproduct of any object in $R_\Prism$ with a transversal object exists in $R_\Prism$. As an immediate generalization, we introduce the so-called relatively quasiregular semiperfectoid covers of $R$, whose prismatic cohomology (of $\delta$-pairs, in the sense of Antieau-Krause-Nikolaus) produces objects in $R_\Prism$ that cover the final object in $\Shv(R_\Prism)$ and admit finite self-coproducts in $R_\Prism$. 
\end{abstract}

\maketitle

{\setcounter{tocdepth}{1}\tableofcontents}

\section{Introduction}\label{section:introduction}

Let $p$ be a fixed prime number and let $R$ be a $p$-quasisyntomic ring. In this note, we provide conditions for objects in the absolute prismatic site $R_\Prism$ to cover the final object and admit finite self-coproducts in $R_\Prism$.

\subsection{Main Results}\label{subsection:main-results}

In Section \ref{section:transversal-objects-and-coproducts-in-absolute-prismatic-sites}, we explain the following result.
\begin{thm}[Corollary {\ref{cor:coproducts-in-RPrism}}]\label{thm:main-result-transversal-cover}
    Let $R$ be a $p$-quasisyntomic ring, and suppose that $(A,I,u)\in R_\Prism$ is a transversal object, in the sense that the following conditions are satsified:
    \begin{itemize}
        \item[(1)] The ring $A$ is $p$-torsion-free. 

        \item[(2)] The map $u:R\rightarrow A/I$ is $p$-completely faithfully flat. 

        \item[(3)] The cotangent complex $L\Omega^1_{(A/I)/(A\otimes R)}\in \rmD(A/I)$ has $p$-complete Tor-amplitude concentrated in homological degree $1$, where $A/I$ is regarded as an $A\otimes R$-algebra via the canonical map $A\rightarrow A/I$ and the map $u:R\rightarrow A/I$, using the universal property of the tensor product. 
    \end{itemize}
    Then $(A,I,u)$ covers the final object in $\Shv(R_\Prism)$ and admits finite self-coproducts in $R_\Prism$. As a consequence, we have an equivalence of categories 
    $$
    \hatD_\crys(R_\Prism,\calO_\Prism)\xrightarrow{\sim}\varprojlim_{\bullet\in \Delta}\hatD(A^\bullet),
    $$
    where $\hatD_\crys(R_\Prism,\calO_\Prism)$ is the derived category of prismatic crystals over $R$, and $(A^\bullet,I^\bullet,u^\bullet)$ is the cosimplicial object obtained by taking self-coproducts of $(A,I,u)\in R_\Prism$, and for each $n\geq 0$, $\hatD(A^n)$ is the full subcategory of $\rmD(A^n)$ spanned by those complexes of $A^n$-modules which are $(p,I^n)$-complete.
\end{thm}

In fact, we provide an immediate generalization of this result in Section \ref{section:relatively-qrsp-covers}. 

\begin{thm}[Corollary \ref{cor:coproducts-with-relatively-qrsp-covers}]\label{thm:main-results-relatively-qrsp-cover}
    Let $R$ be a $p$-quasisyntomic ring, and let $(A',I')$ be a bounded prism, $R'$ be an $R$-algebra with the structure map $u':R\rightarrow R'$, and let $v':A'/I'\rightarrow R'$ be a ring homomorphism. Suppose that $(A',R',u')$ is a relatively quasiregular semiperfectoid cover of $R$, in the sense that the following conditions are satisfied:
    \begin{itemize}
        \item[(1)] The map $v':A'/I'\rightarrow R'$ is surjective. 

        \item[(2)] The cotangent complex $L\Omega^1_{R'/(A'/I')}\in \rmD(R')$ has $p$-complete Tor-amplitude concentrated in homological degree $1$. 

        \item[(3)] The map $v':R\rightarrow R'$ is $p$-completely faithfully flat. 

        \item[(4)] The ring $A'$ is $p$-torsion-free, and the cotangent complex $L\Omega^1_{R'/(A'\otimes R)}\in \rmD(R')$ has $p$-complete Tor-amplitude concentrated in homological degree $1$ 
    \end{itemize}
    Then $(\Prism_{R'/A'},I'\Prism_{R'/A'})$, naturally regarded as an object in $R_\Prism$, covers the final object and admits finite self-coproducts. 
\end{thm}

%Although logically Section \ref{section:transversal-objects-and-coproducts-in-absolute-prismatic-sites} can be subsumed under Section \ref{section:relatively-qrsp-covers}, we choose to maintain Section \ref{section:transversal-objects-and-coproducts-in-absolute-prismatic-sites} as a trace of mathematical movement.

\subsection{Main Idea}\label{subsection:main-idea}

The main idea of this note comes from \cite[Proposition 2.4.5]{BL22} and \cite[Proposition 2.14]{BS23}. Generally speaking, for a $p$-quasisyntomic ring $R$ and an object $(A,I,u)\in R_\Prism$, we ask when, for any $(B,J,v)\in R_\Prism$, the coproduct $(A,I,u)\coprod (B,J,v)$ exists in $R_\Prism$. For some concrete $(A,I,u)$, such a coproduct is usually constructed as a prismatic envelope; see, for example, \cite[Example 2.6 (1)]{BS23}. In the general situation, we construct such a coproduct as certain relative prismatic cohomology objects; roughly, the underlying $\delta$-ring of $(A,I,u)\coprod (B,J,v)$ is given by $\Prism_{(A/I)\widehat{\otimes}^L_R (B/J)/A\otimes^L B}$. We need the following Theorem. 

\begin{thm}[Bhatt--Scholze, {\cite[Proposition 7.10]{BS22}}, Theorem \ref{thm:prismatic-cohomology-of-prismatic-relative-quasiregular-semiperfectoids}]
    Let $(A,I)$ be a bounded prism, and let $R$ be an $A/I$-algebra such that $A/I\rightarrow R$ is surjective and $L\Omega^1_{R/(A/I)}$ has $p$-complete Tor-amplitude concentrated in homological degree $1$. Then $\Prism_{R/A}$ is concentrated in degree $0$, and the pair $(\Prism_{R/A},I\Prism_{R/A})$ gives an initial object in the category $(R/A)_\Prism$ with a map $\iota:R\rightarrow \Prism_{R/A}/I\Prism_{R/A}$. Moreover, the map $R\rightarrow \Prism_{R/A}/I\Prism_{R/A}$ is $p$-completely faithfully flat.

    If more specifically $R=A/(I,r_1,r_2,\ldots)$ where $r_1,r_2,\ldots\in A/I$ is a regular sequence, then we have 
    $$
    \Prism_{R/A}\simeq A\{\frac{r_1}{I},\frac{r_2}{I},\ldots\}^\wedge_{(p,I)}.
    $$
\end{thm}

We also adopt the notion of relatively quasiregular semiperfectoids from \cite{AKN23}, and we use the prismatic cohomology relative to $\delta$-rings developed in \cite{AKN23} for ease of exposition, although all we actually need is the aforementioned theorem of Bhatt--Scholze. Finally, this note provides a site-theoretic approach to some results in \cite{AKN23}; see, in particular, \cite[Lemma 4.20]{AKN23} together with \cite[Theorem 7.17]{BL22schemes}.

\subsection{Motivation}\label{subsection:motivation}

This note grew out of an attempt to compute prismatic cohomology (as well as syntomic cohomology) by choosing suitable covers and handling the corresponding descent. There are many existing works related to this line of thought. For example, in \cite{LW22} there is a computation of $\TC(\calO_K;\F_p)$ using essentially Breuil-Kisin prisms in the form of relative topological Hochschild homology; in \cite{AKN24} there is a delicate algorithm that computes the syntomic cohomology of $\Z/p^n$ (in order to compute $\AlgK(\Z/p^n)$) by descent; in \cite{BS23}, \cite{GLQ25}, and \cite{GMW23}, among others, there are studies of prismatic $F$-crystals and Hodge-Tate crystals by descent techniques. In future work, we will use Examples \ref{ex:qdR-polynomial-cover} and \ref{ex:complete-intersection-Zpzetap} (and their variants) to compute the prismatic cohomology (as well as the syntomic cohomology) and study the corresponding crystals.%similar considerations can be found in \cite{ZeyuLiu25}.

\subsection*{Acknowledgements}
The author is very grateful to Guozhen Wang for suggesting this investigation. The author would like to thank Akiyoshi Sannai for pointing out that, in an early version of this article, Lemmas 4.4 and 4.5 were missing some hypotheses. The author also thanks Hui Gao, Heng Du, Yu Min, Yupeng Wang, and Wei Yang for valuable discussions.

\section{A Reminder on Absolute Prismatic Sites}\label{section:a-reminder-on-absolute-prismatic-sites}

Let $p$ be a fixed prime number. In this section, we recall the standard definitions concerning absolute prismatic sites for $p$-quasisyntomic rings. For more details on prismatic cohomology, see \cite{BS22} and \cite{BL22}. 

\begin{defn}[$p$-Quasisyntomic Rings, {\cite[Definition C.6]{BL22}}]\label{defn:quasisyntomic-rings}
    We say that a commutative ring $R$ is {\it $p$-quasisyntomic} if it satisfies the following conditions:
    \begin{itemize}
        \item[(1)] The commutative ring $R$ has bounded $p$-power torsion.

        \item[(2)] The object $L\Omega^1_{R}\in \rmD(R)$ has $p$-complete Tor-amplitude concentrated in homological degree $[0,1]$.
    \end{itemize}

    %A $p$-quasisyntomic ring $R$ is called {\it $p$-quasismooth} if it is $p$-torsion-free and $L\Omega^1_R\in \rmD(R)$ has $p$-complete Tor-amplitude concentrated in degree $[0,0]$.

    %A $p$-quasisyntomic ring $R$ is called {\it $p$-quasi-\'etale} if it is $p$-torsion-free and $L\Omega^1_R\otimes^L_R(R/pR)\simeq 0$.

    %A ring homomorphism $f:R\rightarrow S$ is called {\it $p$-quasisyntomic} if it is $p$-completely flat and $L\Omega^1_{S/R}\in\rmD(S)$ has $p$-complete Tor-amplitude concentrated in homological degree $[0,1]$. A $p$-quasisyntomic homomorphism is called a {\it $p$-quasisyntomic cover} if it is $p$-completely faithfully flat.
\end{defn}

\begin{ex}\label{ex:quasisyntomic-rings}
    The following commutative rings are $p$-quasisyntomic:
    \begin{itemize}
        \item[(1)] Let $k$ be a perfect ring of characteristic $p$. Then $W(k)$, the ring of Witt vectors with coefficients in $k$, is $p$-quasisyntomic. 

        \item[(2)] A polynomial ring $W(k)[x_1,\ldots,x_n]$ is $p$-quasisyntomic. 

        \item[(3)] Let $f_1,\ldots,f_r\in W(k)[x_1,\ldots,x_n]$ be a regular sequence. Then $W(k)[x_1,\ldots,x_n]/(f_1,\ldots,f_r)$ is $p$-quasisyntomic. 

        \item[(4)] Let $K$ be a finite extension of $\Q_p$. Then the ring of integers $\calO_K$ is $p$-quasisyntomic. 
    \end{itemize}
\end{ex}

To introduce the prismatic sites, let us first recall the notion of prisms.

\begin{defn}[Prisms, {\cite[Definition 3.2]{BS22}}]\label{defn:prisms}
    A {\it prism} is a pair $(A,I)$, where $A$ is a $\delta$-ring and $I\subset A$ an ideal, such that the following conditions are satisfied:
    \begin{itemize}
        \item[(1)] $I\subset A$ defines a Cartier divisor on $\Spec(A)$. 

        \item[(2)] $A$ is derived $(p,I)$-complete.

        \item[(3)] $p\in I+\phi(I)A$, where $\phi$ is the Frobenius associated to the $\delta$-ring $A$. 
    \end{itemize}

    A prism $(A,I)$ is called {\it bounded} if $A/I$ has bounded $p$-power torsion. A map $(A,I)\rightarrow (B,J)$ of prisms is {\it (faithfully) flat} if the map $A\rightarrow B$ is $(p,I)$-completely (faithfully) flat.
\end{defn}

We now introduce the absolute prismatic site for a $p$-quasisyntomic $R$. In fact, the absolute prismatic site can be defined for any bounded $p$-adic formal scheme. However, the derived prismatic cohomology coincides with the site-theoretic prismatic cohomology only for $p$-quasisyntomic $p$-adic formal schemes; see, for example, \cite[Theorem 4.4.30]{BL22}. 

\begin{defn}[Absolute Prismatic Sites, {\cite[Definition 4.4.27]{BL22}}]\label{defn:absolute-prismatic-sites}
    Let $R$ be a $p$-quasisyntomic ring. We define a category $R_\Prism$ as follows:
    \begin{itemize}
        \item The objects of $R_\Prism$ are triples $(A,I,u)$, where $(A,I)$ is a bounded prism and $u:R\rightarrow A/I$ is a ring homomorphism. 

        \item A morphism from $(A,I,u)$ to $(B,J,v)$ in $R_\Prism$ is a morphism of prisms $(A,I)\rightarrow (B,J)$ for which the composite map 
        $$
        R\xrightarrow{u}A/I\rightarrow B/J
        $$
        is equal to $v$.
    \end{itemize}
    We refer to $R_\Prism$ as the {\it absolute prismatic site of $R$}. We define the {\it flat topology} on $(R_\Prism)^\op$ to be the Grothendieck topology generated by those finite collections of morphisms 
    $$
    \{(A,I,u)\rightarrow (A_s,I_s,u_s)\}_{s\in S}
    $$
    for which the induced ring homomorphism $A\rightarrow \prod_{s\in S}A_s$ is $(p,I)$-completely faithfully flat. Finally, the {\it absolute prismatic cohomoloyg of $R$} is defined as 
    $$
    \Prism_R:=\varprojlim_{(A,I,u)\in R_\Prism}A.
    $$
\end{defn}

We recall the notion of absolute prismatic crystals over $R$:

\begin{defn}[Absolute Prismatic Crystals, {\cite[Proposition 8.15]{BL22schemes}}]\label{defn:absolute-prismatic-crystals}
    Let $R$ be a $p$-quasisyntomic ring. We define the category $\hatD_\crys(R_\Prism,\calO_\Prism)$ to be the following inverse limit:
    $$
    \hatD_\crys(R_\Prism,\calO_\Prism):=\varprojlim_{(A,I,u)\in R_\Prism}\hatD(A),
    $$
    where $\hatD(A)$ denotes the full subcategory of $\rmD(A)$ spanned by those complexes which are $(p,I)$-complete. We refer to $\hatD_\crys(R_\Prism,\calO_\Prism)$ as the {\it category of absolute prismatic crystals over $R$}. 

    Similarly, we define the category $\Vect(R_\Prism,\calO_\Prism)$ to be the following inverse limit 
    $$
    \Vect(R_\Prism,\calO_\Prism):=\varprojlim_{(A,I,u)\in R_\Prism} \Vect(A),
    $$
    where $\Vect(A)$ is the category of finite projective $A$-modules. We refer to $\Vect(R_\Prism,\calO_\Prism)$ as the {\it category of absolute prismatic vector bundles over $R$}. 
\end{defn}

Note that $\hatD(-)$ and $\Vect(-)$ satisfy flat descent.

\begin{prop}\label{prop:faithfully-flat-descent}
    Let $R$ be a $p$-quasisyntomic ring. Then the assignments 
    $$
    (A,I,u)\in R_\Prism \mapsto \hatD(A)\quad \quad \text{and}\quad \quad  (A,I,u)\in R_\Prism\mapsto \Vect(A)
    $$
    are sheaves for the flat topology on $(R_\Prism)^\op$.
\end{prop}

\begin{proof}
    See \cite[Appendix A]{ALB19} for the proof in the case of vector bundles. The argument for $\hatD(-)$ is similar. See also the discussion under \cite[Theorem 2.2]{BS23}.
\end{proof}

In fact, the inverse limit $\varprojlim_{(A,I,u)\in R_\Prism}$ can be computed by the totalization of a certain cosimplicial diagram. More precisely, we can find covers of the final object in the topos $\Shv(R_\Prism)$. This fact is already known; see \cite[Proposition 2.14]{BS23}, where covers of the final object of $\Shv(R_\Prism)$ are produced by the prismatic cohomology of quasiregular semiperfectoids that cover $R$. See also \cite[Lemma 3.3.10]{BL22} and \cite[Example 2.8]{BS23}. In this note, we will provide conditions for objects in $R_\Prism$ to cover the final object, and we also produce these covers via the prismatic cohomology of relatively quasiregular semiperfectoid $\delta$-pairs; see Section \ref{section:relatively-qrsp-covers}. 

\begin{defn}\label{defn:cover-the-final-object}
    Let $R$ be a $p$-quasisyntomic ring. We say that an object $(A,I,u)\in R_\Prism$ {\it covers the final object in $\Shv(R_\Prism)$} if the following condition is satisfied:
    \begin{itemize}
        \item For any object $(B,J,v)\in R_\Prism$, there exist $(B',J',v')\in R_\Prism$, a $(p,J)$-completely faithfully flat map $(B,J,v)\rightarrow (B',J',v')$, and a map $(A,I,u)\rightarrow (B',J',v')$.
    \end{itemize}
\end{defn}

\begin{lem}\label{lem:cover-the-final-object}
    Let $R$ be a $p$-quasisyntomic ring. Assume that $(B,J,v)\in R_\Prism$ covers the final object and admits finite self-coproducts in $R_\Prism$. Let $(B^{\bullet},JB^\bullet,v^\bullet)$ denote the cosimplicial object in $R_\Prism$ obtained by taking finite self-coproducts of $(B,J,v)$. Then we have the following equivalences of categories
    $$
    \hatD_\crys(R_\Prism,\calO_\Prism)=\varprojlim_{(A,I,u)\in R_\Prism}\hatD(A)\xrightarrow{\sim}\varprojlim_{\bullet\in \Delta}\hatD(B^\bullet),
    $$
    and 
    $$
    \Vect(R_\Prism,\calO_\Prism)=\varprojlim_{(A,I,u)\in R_\Prism} \Vect(A)\xrightarrow{\sim}\varprojlim_{\bullet\in \Delta}\Vect(B^\bullet).
    $$
\end{lem}

\begin{proof}
    This is immediate by flat descent and cofinality.
\end{proof}

\section{Transversal Objects and Coproducts in Absolute Prismatic Sites}\label{section:transversal-objects-and-coproducts-in-absolute-prismatic-sites}

In this section, we follow the approach of \cite[$\S 2.1$ and $\S 2.5$]{BL22} and introduce the notion of transversal objects in the absolute prismatic site of a $p$-quasisyntomic ring. The notion of {\it transversal prisms} was first introduced in \cite[Lemma 2.1.7]{ALB19}; see also \cite{ALB20}.

\begin{defn}[Transversal Objects]\label{defn:transversal-objects}
    Let $R$ be a $p$-quasisyntomic ring. An object $(A,I,u)\in R_\Prism$ is called {\it transversal} if it satisfies the following conditions:
    \begin{itemize}
        \item[(1)] The ring $A$ is $p$-torsion-free. 

        \item[(2)] The map $u:R\rightarrow A/I$ is $p$-completely flat. 

        \item[(3)] The cotangent complex $L\Omega^1_{(A/I)/(A\otimes R)}\in \rmD(A/I)$ has $p$-complete Tor-amplitude concentrated in homological degree $1$, where $A/I$ is regarded as an $A\otimes R$-algebra via the canonical map $A\rightarrow A/I$ and the map $u:R\rightarrow A/I$, using the universal property of the tensor product. 
    \end{itemize}

    A transversal object $(A,I,u)$ is called a {\it transversal cover} if $u:R\rightarrow A/I$ is $p$-completely faithfully flat.
\end{defn}

\begin{rem}\label{rem:transversal-flat-morphism}
    Let $R$ be a $p$-quasisyntomic ring and $(A,I,u)\in R_\Prism$ be a transversal object. Suppose that $(A,I,u)\rightarrow (B,J,v)$ is a $(p,I)$-completely flat morphism in $R_\Prism$. Then $(B,J,v)$ is also a transversal object. It suffices to note that 
    $$
    L\Omega^1_{(B/J)/B\otimes R}\simeq L\Omega^1_{(A/I)/A\otimes R}\otimes_A^L B\simeq L\Omega^1_{(A/I)/(A\otimes R)}\otimes^L_{A/I}B/J.
    $$
\end{rem}

\begin{lem}\label{lem:transveral-LOmegaR0}
    Let $R$ be a $p$-quasisyntomic ring. Let $(A,I,u)\in R_\Prism$ be an object such that $A$ is $p$-torsion-free and the map $u:R\rightarrow A/I$ is $p$-completely flat. If $L\Omega^1_R$ has $p$-complete Tor-amplitude concentrated in homological degree $[0,0]$, then $(A,I,u)$ is a transveral object in $R_\Prism$.
\end{lem}

\begin{proof}
    Consider the chain of ring homomorphisms $A\rightarrow A\otimes R\rightarrow A/I$ and the associated cofiber sequence of cotangent complexes
    $$
    L\Omega^1_{(A\otimes R/A)}\otimes^L_{A\otimes R}(A/I)\rightarrow L\Omega^1_{(A/I)/A}\rightarrow L\Omega^1_{(A/I)/(A\otimes R)},
    $$
    and note that 
    $$L\Omega^1_{(A\otimes R)/A}\otimes^L_{A\otimes R}(A/I)\simeq L\Omega^1_R\otimes^L_R(A/I).
    $$ 
    Therefore, $L\Omega^1_{(A/I)/(A\otimes R)}$ has $p$-complete Tor-amplitude in homological degree $1$.
\end{proof}

\begin{ex}\label{ex:transversal-objects-in-ZpPrism}
    An object $(A,I)\in (\Z_p)_\Prism$ (that is, $(A,I)$ is a bounded prism) is a transversal object in the sense of Definition \ref{defn:transversal-objects} if and only if $A/I$ is $p$-torsion-free, that is, if and only if $(A,I)$ is a transversal prism in the sense of \cite[Definition 2.1.3]{BL22}. To confirm this, it suffices to appeal to Lemma \ref{lem:transveral-LOmegaR0}. 
\end{ex}

\begin{lem}\label{lem:transversal-for-u-quasietale}
    Let $R$ be a $p$-quasisyntomic ring. Let $(A,I,u)\in R_\Prism$ be an object such that $A$ is $p$-torsion-free and the map $u:R\rightarrow A/I$ is $p$-completely flat. Then the following statements hold.
    \begin{itemize}
        \item[(1)] If $L\Omega^1_{(A/I)/R}$ has $p$-complete Tor-amplitude concentrated in degree $1$ and $L\Omega^1_A\otimes^L_A(A/I)$ has $p$-complete Tor-amplitude concentrated in degree $[0,0]$, then $(A,I,u)$ is a transversal object in $R_\Prism$. 

        \item[(2)] If $L\Omega^1_{(A/I)/R}$ is $p$-completely zero, then $(A,I,u)$ is a transveral object in $R_\Prism$ if and only if $L\Omega^1_A\otimes^L_A(A/I)$ has $p$-complete Tor-amplitude concentrated in degree $[0,0]$.
    \end{itemize}
\end{lem}

\begin{proof}
    Consider the chain of ring homomorphisms $R\rightarrow A\otimes R\rightarrow A/I$ and the associated cofiber sequence of cotangent complexes
    $$
    L\Omega^1_{(A\otimes R)/R}\otimes^L_{A\otimes R}(A/I)\rightarrow L\Omega^1_{(A/I)/R}\rightarrow L\Omega^1_{(A/I)/(A\otimes R)},
    $$
    and note that
    $$
    L\Omega^1_{(A\otimes R)/R}\otimes^L_{A\otimes R}(A/I)\simeq L\Omega^1_A\otimes^L_A(A/I).
    $$
    Hence, the statements hold.
\end{proof}

\begin{ex}\label{ex:Breuil-Kisin}
    Let $K$ be a finite extension of $\Q_p$, and let $\calO_K$ be the ring of integers of $K$. Choose a uniformizer $\varpi\in \calO_K$, and let $k=\calO_K/\varpi$ be the residue field. Finally, let $E_K(z)\in W(k)[z]$ be a minimal polynomial of $\varpi$. Then the Breuil-Kisin prism $(W(k)[[z]],(E_K))$, naturally regarded as an object in $(\calO_K)_\Prism$, is a transversal object. To confirm this, by Lemma \ref{lem:transversal-for-u-quasietale} (2), it suffices to check that $L\Omega^1_{W(k)[[z]]/\Z}\otimes^L_{W(k)[[z]]}\calO_K$ has $p$-complete Tor-amplitude concentrated in degree $[0,0]$. In fact, we claim that there is a natural equivalence 
    $$
    L\Omega^1_{W(k)[z]/\Z}\otimes^L_{W(k)[z]}\calO_K\xrightarrow{\sim}L\Omega^1_{W(k)[[z]]/\Z}\otimes^L_{W(k)[[z]]}\calO_K.
    $$
    
    First we consider the chain of ring homomorphisms $\Z\rightarrow W(k)[z]\rightarrow W(k)[[z]]$ and the associated cofiber sequence:
    $$
    L\Omega^1_{W(k)[z]/\Z}\otimes^L_{W(k)[z]}W(k)[[z]]\rightarrow L\Omega^1_{W(k)[[z]]/\Z}\rightarrow L\Omega^1_{W(k)[[z]]/W(k)[z]},
    $$
    and by applying $-\otimes_{W(k)[[z]]}\calO_K$ we obtain the following cofiber sequence
    $$
    L\Omega^1_{W(k)[z]/\Z}\otimes^L_{W(k)[z]}\calO_K\rightarrow L\Omega^1_{W(k)[[z]]/\Z}\otimes^L_{W(k)[[z]]}\calO_K\rightarrow L\Omega^1_{W(k)[[z]]/W(k)[z]}\otimes^L_{W(k)[[z]]}\calO_K.
    $$

    Then we consider the chain of ring homomorphisms $W(k)[z]\rightarrow W(k)[[z]]\rightarrow \calO_K$. Its associated cofiber sequence implies that
    $$
    L\Omega^1_{W(k)[[z]]/W(k)[z]}\otimes^L_{W(k)[[z]]}\calO_K\simeq 0,
    $$
    and therefore $L\Omega^1_{W(k)[[z]]/\Z}\otimes^L_{W(k)[[z]]}\calO_K\simeq L\Omega^1_{W(k)[z]/\Z}\otimes^L_{W(k)[z]}\calO_K$ which has $p$-complete Tor-amplitude concentrated in degree $[0,0]$.
\end{ex}

In the category $R_\Prism$, coproducts with transversal objects exist. In \cite[Proposition 2.4.5]{BL22}, the existence of such coproducts in $(\Z_p)_\Prism$ is verified using the prismatic envelopes. However, the standard formalism of prismatic envelopes (for example, \cite[Proposition 3.13]{BS22}) is not sufficiently convenient for a general $p$-quasisyntomic ring. Instead, we appeal to the theory of relative prismatic cohomology developed in \cite{BS22}, and realize the expected coproducts in $R_\Prism$ as certain relative prismatic cohomology objects (see Construction \ref{construction:coproducts-in-RPrism}). For ease of exposition, we adopt the notion of relatively quasiregular semiperfectoids and the prismatic cohomology relative to $\delta$-rings, both of which are developed in \cite{AKN23}, although all we actually need is contained in \cite{BS22}; see, for example, \cite[page 49, (\$)]{AKN23}.

\begin{defn}[$\delta$-Pairs, {\cite[Definition 2.1]{AKN23}}]\label{defn:delta-pairs}
    A {\it $\delta$-pair} is a pair $(A,R)$, where $A$ is a $\delta$-ring and $R$ is a commutative $A$-algebra. 

    A $\delta$-pair $(A,R)$ is called 
    \begin{itemize}
        \item bounded if both $A$ and $R$ have bounded $p$-power torsion;

        \item pre-prismatic if the kernel of $A\rightarrow R$ contains a Cartier divisor $I$ such that Zariski locally on $\Spec A$ any generator $d$ of $I$ has the property that $\delta(d)$ maps to a unit in $R^\wedge_p$;

        \item prismatic if it is pre-prismatic as exhibited by an ideal $I\subset A$ and $(A,I)$ is a prism.
    \end{itemize}
\end{defn}

\begin{defn}[Prismatic Sites Relative to $\delta$-Rings, {\cite[Definition 2.4]{AKN23}}]\label{defn:prismatic-sites-relative-to-delta-rings}
    Let $(A,R)$ be a $\delta$-pair. We define a category $(R/A)_\Prism$ as follows:
    \begin{itemize}
        \item The objects of $(R/A)_\Prism$ are quadruples $(B,J,u,u')$, where $(B,J)$ is a bounded prism and $u:R\rightarrow B/J$ is a ring homomorphism, and $u':A\rightarrow B$ is a map of $\delta$-rings, such that the following diagram is commutative
        $$
        \xymatrix@R=50pt@C=50pt{
        A \ar[r]^{u'} \ar[d] & B \ar[d] \\
        R\ \ar[r]^u       & B/J
        }
        $$

        \item The morphisms in $(R/A)_\Prism$ are naturally defined so that everything is compatible.
    \end{itemize}
    We refer to $(R/A)_\Prism$ as the {\it prismatic site of $R$ relative to $A$}. Similarly to Definition \ref{defn:absolute-prismatic-sites}, we define the {\it flat topology} on $(R/A)_\Prism^\op$ to be the $(p,J)$-completely faithfully flat topology in $B$. Finally, the {\it prismatic cohomoloyg of $R$ relative to $A$} is defined as 
    $$
    \Prism_{R/A}:=\varprojlim_{(B,J,u,u')\in (R/A)_\Prism}B.
    $$
    Also, we define 
    $$
    I\Prism_{R/A}:=\varprojlim_{(B,J,u,u')\in (R/A)_\Prism}J.
    $$
\end{defn}

\begin{defn}[Relatively Quasiregular Semiperfectoids, {\cite[Definition 9.2]{AKN23}}]\label{defn:relatively-quasiregular-semiperfectoids}
    A bounded $\delta$-pair $(A,R)$ is called {\it relatively quasiregular semiperfectoid} if the following conditions are satisfied:
    \begin{itemize}
        \item[(1)] $R$ is $p$-complete.

        \item[(2)] There is a factorization $A\rightarrow A'\rightarrow R$ where $A\rightarrow A'$ is a $p$-completely relatively perfect map of $\delta$-rings,
        and the $\delta$-pair $(A',R)$ is pre-prismatic as exhibited by an ideal $I\subset A'$, and the map $A'\rightarrow R$ is surjective, and $L\Omega^1_{R/(A'/I)}\in \rmD(R)$ has $p$-complete Tor-amplitude concentrated in homological degree $1$.
    \end{itemize}
    Note that we say a map of $\delta$-rings $B\rightarrow B'$ is $p$-completely relatively perfect if the canonical map $B\otimes^L_{\varphi_B,B}B'\xrightarrow{\varphi_{B'/B}}B'$ is a $p$-complete equivalence. In particular, $\varphi_B$ and $\varphi_{B'}$ together induce the following $p$-complete equivalence:
    $$
    L\Omega^1_{B'/B}\otimes^L_{B,\varphi_B}B\simeq L\Omega^1_{(B\otimes^L_{\varphi_B,B}B')/B}\rightarrow L\Omega^1_{B'/B},
    $$
    which, however, is a zero map after applying $-\otimes^L_{B'}B'/p$ to it. Therefore, $L\Omega^1_{B'/B}$ is $p$-completely zero.
\end{defn}

\begin{rem}\label{rem:relatively-quasiregular-semiperfectoids}
    Let $(A,R)$ be a relatively quasiregular semiperfectoid $\delta$-pair, and $A\rightarrow A'\rightarrow R$ a factorization as in Defintion \ref{defn:relatively-quasiregular-semiperfectoids} (2). Then $L\Omega^1_{R/A'}\in \rmD(R)$ has $p$-complete Tor-amplitude concentrated in homological degree $1$; it suffices to consider the chain of ring homomorphisms $A'\rightarrow A'/I\rightarrow R$ and the associated cofiber sequence of cotangent complexes
    $$
    L\Omega^1_{(A'/I)/A'}\otimes^L_{A'/I}R\rightarrow L\Omega^1_{R/A'}\rightarrow L\Omega^1_{R/(A'/I)}.
    $$
    Moreover, since $A\rightarrow A'$ is $p$-completely relatively perfect and, in particular, $L\Omega^1_{A'/A}$ is $p$-completely zero, we see that $L\Omega^1_{R/A}$ also has $p$-complete Tor-amplitude concentrated in degree $1$.
\end{rem}

We need the following result. 

\begin{thm}[Bhatt-Scholze, {\cite[Theorem 9.6]{AKN23}}]\label{thm:prismatic-cohomology-of-relative-quasiregular-semiperfectoids}
    Let $(A,R)$ be a relatively quasiregular semiperfectoid $\delta$-pair. Then $\Prism_{R/A}$ is concentrated in degree $0$, and the pair $(\Prism_{R/A},I\Prism_{R/A})$ gives an initial object in the category $(R/A)_\Prism$, together with a map $\iota:R\rightarrow \Prism_{R/A}/I\Prism_{R/A}$. Moreover, the map $R\rightarrow \Prism_{R/A}/I\Prism_{R/A}$ is $p$-completely faithfully flat.

    If there is a factorization $A\rightarrow A'\rightarrow R$ as in Definition \ref{defn:relatively-quasiregular-semiperfectoids} (2), with $(A',I)$ a prism, then 
    $$
    \Prism_{R/A}\simeq \Prism_{R/A'},
    $$
    where $\Prism_{R/A'}$, which is the prismatic cohomology for the $\delta$-pair $(A',R)$, coincides with the prismatic cohomology of $R'$ relative to $(A',I)$ defined in \cite{BS22}.
    
    More specifically, if $R=A'/(I,r_1,r_2,\ldots)$, where $r_1,r_2,\ldots\in A'/I$ form a regular sequence, then
    $$
    \Prism_{R/A}\simeq \Prism_{R/A'}\simeq A'\{\frac{r_1}{I},\frac{r_2}{I},\ldots\}^\wedge_{(p,I)}.
    $$
\end{thm}

\begin{proof}
     If $A\rightarrow A'\rightarrow R$ is a factorization as in Definition \ref{defn:relatively-quasiregular-semiperfectoids} (2), then by \cite[Theorem 1.2 (6)]{AKN23}, there is a canonical equivalence $\Prism_{R/A}\xrightarrow{\sim}\Prism_{R/A'}$. Note that $L\Omega^1_{A'/A}$ is $p$-completely zero, and therefore the $p$-complete formal lifting problem for $A\rightarrow A'$ can be solved uniquely. Thus we can assume that the $\delta$-pair $(A,R)$ is prismatic, in the sense that $(A,I)$ is a bounded prism and $R$ is an $A/I$-algebra. It remains to verify Theorem \ref{thm:prismatic-cohomology-of-prismatic-relative-quasiregular-semiperfectoids}. 
\end{proof}

\begin{thm}[{\cite[Proposition 7.10 and Theorem 15.2]{BS22}}]\label{thm:prismatic-cohomology-of-prismatic-relative-quasiregular-semiperfectoids}
    Let $(A,I)$ be a bounded prism, and let $R$ be an $A/I$-algebra such that $A/I\rightarrow R$ is surjective and $L\Omega^1_{R/(A/I)}$ has $p$-complete Tor-amplitude concentrated in homological degree $1$. Then $\Prism_{R/A}$ is concentrated in degree $0$, and the pair $(\Prism_{R/A},I\Prism_{R/A})$ gives an initial object in the category $(R/A)_\Prism$, together with a map $\iota:R\rightarrow \Prism_{R/A}/I\Prism_{R/A}$. Moreover, the map $R\rightarrow \Prism_{R/A}/I\Prism_{R/A}$ is $p$-completely faithfully flat.
\end{thm}

\begin{proof}
    The argument is essentially the same as that of \cite[Proposition 7.10]{BS22}. 

    Since $L\Omega^1_{R/(A/I)}[-1]$ is $p$-completely flat, the derived $p$-completions of all $L\Omega^i_{R/(A/I)}[-i]$ are $p$-completely flat $R$-modules and therefore discrete. Thus the Hodge-Tate comparison implies that $\overline{\Prism}_{R/A}$ is discrete and $p$-completely faithfully flat over $R$. By \cite[Lemma 7.7 (2)]{BS22}, the pair $(\Prism_{R/A},I\Prism_{R/A})$ defines an object in $(R/A)_\Prism$. It remains to check that $\Prism_{R/A}$ is initial, or equivalently that the idempotent endomorphism from \cite[Lemma 7.7 (3)]{BS22} is the identity. 

    Note that, by assumption $A/I\rightarrow R$ is surjective. Choose a generating set $\{f_t\}_{t\in T}$ of the kernel of $A/I\rightarrow R$. Then the resulting map 
    $$
    R'=A/I\langle X_t^{1/p^\infty}\mid t\in T\rangle/(X_t-f_t)^\wedge\rightarrow R,\quad \quad X_t\mapsto 0
    $$
    is surjective and also induces a surjection on cotangent complexes. Thus, by the Hodge-Tate comparison, $\Prism_{R'/A}\rightarrow \Prism_{R/A}$ is surjective, so it suffices to show that the idempotent endomorphism of $\Prism_{R'/A}$ is the identity. However, by truncating the set $T$ of generators, we can express $R'$ as a filtered colimit of regular quotients of the form 
    $$
    P=A/I\langle X_1^{1/p^\infty},\ldots,X_r^{1/p^\infty}\rangle/(X_1-f_1,\ldots,X_r-f_r),
    $$
    so it suffices to prove that the idempotent endomorphism of $\Prism_{P/A}$ is the identity for each such $P$. Now if $A'=A\langle X_1^{1/p^\infty},\ldots,X_r^{1/p^\infty}\rangle^\wedge_{(p,I)}$, then we also have $\Prism_{P/A}\simeq \Prism_{P/A'}$ compatibly with the idempotent endomorphisms. By \cite[Example 7.9]{BS22}, the idempotent endomorphism of $\Prism_{P/A'}$ is the identity. 
\end{proof}

Now we are ready to construct coproducts in $R_\Prism$. 

\begin{lem}\label{lem:coproducts-in-RPrism}
    Let $R$ be a $p$-quasisyntomic ring, and let $(A,I,u),(B,J,v)\in R_\Prism$, with $(A,I,u)$ a transversal object in the sense of Definition \ref{defn:transversal-objects}. Set $\overline{A}:=A/I$, and $\overline{B}:=B/J$. Then the $\delta$-pair $(A\otimes B,\overline{A}\widehat{\otimes}^L_R\overline{B})$ is a relatively quasiregular semiperfectoid $\delta$-pair.
\end{lem}

\begin{proof}
    The ideal $A\otimes JB$ exhibits $(A\otimes B,\overline{A}\widehat{\otimes}^L_R\overline{B})$ as a pre-prismatic $\delta$-pair, and it remains to check that $L\Omega^1_{(\overline{A}\otimes^L_R \overline{B})/(A\otimes \overline{B})}$ (here we utilize the assumption that $R\rightarrow \overline{A}$ is $p$-completely flat) has $p$-complete Tor-amplitude concentrated in homological degree $1$. In fact, we have 
    $$
    L\Omega^1_{(\overline{A}\otimes^L_R \overline{B})/(A\otimes \overline{B})}
    \simeq L\Omega^1_{\overline{A}/(A\otimes R)}\otimes_R^L\overline{B},
    $$
    and since $(A,I,u)$ is a transversal object, by definition, $L\Omega^1_{\overline{A}/(A\otimes R)}\in \rmD(\overline{A})$ has $p$-complete Tor-amplitude concentrated in homological degree $1$, and therefore $L\Omega^1_{(\overline{A}\otimes^L_R \overline{B})/(A\otimes \overline{B})}\in \rmD(\overline{A}\otimes^L_R\overline{B})$ has $p$-complete Tor-amplitude concentrated in homological degree $1$.

\end{proof}

\begin{construction}\label{construction:coproducts-in-RPrism}
    Let $R$ be a $p$-quasisyntomic ring, and let $(A,I,u),(B,J,v)\in R_\Prism$, where $(A,I,u)$ is a transversal object. Set $\overline{A}:=A/I$, and $\overline{B}:=B/J$. Consider the $\delta$-pair $(A\otimes B,\overline{A}\widehat{\otimes}^L_R\overline{B})$. By Lemma \ref{lem:coproducts-in-RPrism}, $(A\otimes B,\overline{A}\widehat{\otimes}^L_R\overline{B})$ is a relatively quasiregular semiperfectoid $\delta$-pair, and by Theorem \ref{thm:prismatic-cohomology-of-relative-quasiregular-semiperfectoids} (in fact Theorem \ref{thm:prismatic-cohomology-of-prismatic-relative-quasiregular-semiperfectoids} is sufficient), we have an initial object $$(\Prism_{(\overline{A}\widehat{\otimes}^L_R\overline{B})/(A\otimes B)},J\Prism_{(\overline{A}\widehat{\otimes}^L_R\overline{B})/(A\otimes B)},\iota)\in ((\overline{A}\widehat{\otimes}^L_R\overline{B})/(A\otimes B))_\Prism.
    $$ 
    Consider the map
    $$
    \iota_R:R\rightarrow \overline{A}\widehat{\otimes}^L_R\overline{B}\xrightarrow{\iota}\Prism_{(\overline{A}\widehat{\otimes}^L_R\overline{B})/(A\otimes B)}/J.
    $$
    We then obtain an object 
    $$
    (\Prism_{(\overline{A}\widehat{\otimes}^L_R\overline{B})/(A\otimes B)},J\Prism_{(\overline{A}\widehat{\otimes}^L_R\overline{B})/(A\otimes B)},\iota_R)\in R_\Prism.
    $$
    We claim that this object is a coproduct of $(A,I,u)$ and $(B,J,v)$ in $R_\Prism$. 
    
    By construction, the object  $(\Prism_{(\overline{A}\widehat{\otimes}^L_R\overline{B})/(A\otimes B)},J\Prism_{(\overline{A}\widehat{\otimes}^L_R\overline{B})/(A\otimes B)},\iota_R)$ naturally receives maps from $(A,I,u)$ and $(B,J,v)$. It remains to check the universal property. Suppose that $(C,K,w)\in R_\Prism$ receives maps from $(A,I,u)$ and $(B,J,v)$. Then $(C,K)$ can naturally be regarded as an object in $((\overline{A}\widehat{\otimes}^L_R\overline{B})/(A\otimes B))_\Prism$. Since $(\Prism_{(\overline{A}\widehat{\otimes}^L_R\overline{B})/(A\otimes B)},J\Prism_{(\overline{A}\widehat{\otimes}^L_R\overline{B})/(A\otimes B)},\iota)$ is initial in $((\overline{A}\widehat{\otimes}^L_R\overline{B})/(A\otimes B))_\Prism$, there is a unique morphism in 
    $$
    (\Prism_{(\overline{A}\widehat{\otimes}^L_R\overline{B})/(A\otimes B)},J\Prism_{(\overline{A}\widehat{\otimes}^L_R\overline{B})/(A\otimes B)},\iota_R)\rightarrow (C,K,w),
    $$
    which justifies the desired universal property. 
\end{construction}

\begin{prop}\label{prop:coproducts-in-RPrism}
    Let $R$ be a $p$-quasisyntomic ring, and let $(A,I,u),(B,J,v)\in R_\Prism$, where $(A,I,u)$ is a transversal object. Then there exists a coproduct of $(A,I,u)$ and $(B,J,v)$ in the category $R_\Prism$. Moreover, the following hold: 
    \begin{itemize}
        \item[(1)] The canonical map $(B,J,v)\rightarrow (A,I,u)\coprod (B,J,v)$ is $(p,J)$-completely flat;

        \item[(2)] If $(A,I,u)$ is a transversal cover, then the canonical map $(B,J,v)\rightarrow (A,I,u)\coprod (B,J,v)$ is $(p,J)$-completely faithfully flat.
    \end{itemize}
\end{prop}

\begin{proof}
    The existence of such a coproduct is explained in Construction \ref{construction:coproducts-in-RPrism}. Let $(C,K,w)$ denote a coproduct of $(A,I,u)$ and $(B,J,v)$. By Construction \ref{construction:coproducts-in-RPrism}, the canonical map $(B,J,v)\rightarrow (C,K,w)$ induces
    $$
    B/J\rightarrow (A/I)\widehat\otimes_R(B/J)\rightarrow C/K,
    $$
    where the first morphism is $p$-completely flat (resp. $p$-completely faithfully flat if $(A,I,u)$ is a transversal cover), and the second morphism is $p$-completely faithfully flat.
\end{proof}

\begin{cor}\label{cor:coproducts-in-RPrism}
    Let $R$ be a $p$-quasisyntomic ring, and let $(A,I,u)\in R_\Prism$ be a transversal cover. Then $(A,I,u)$ covers the final object (in the sense of Definition \ref{defn:cover-the-final-object}) and admits finite self-coproducts in $R_\Prism$. 
\end{cor}

\begin{proof}
    This follows immediately from Proposition \ref{prop:coproducts-in-RPrism}.
\end{proof}

\section{Relatively Quasiregular Semiperfectoid Covers}\label{section:relatively-qrsp-covers}

In this section, we adopt the notion of relatively quasiregular semiperfectoids from \cite[$\S 4$ and $\S 9$]{AKN23} and generalize the treatment in Section \ref{section:transversal-objects-and-coproducts-in-absolute-prismatic-sites} by introducing the notion of relatively quasiregular semiperfectoid covers in the absolute prismatic site of a $p$-quasisyntomic ring.

\begin{defn}[Relatively Quasiregular Semiperfectoid Covers]\label{defn:relatively-quasiregular-semiperfectoid-covers}
    Let $R$ be a $p$-quasisyntomic ring. A triple $(A',R',u')$, where $(A',R')$ is a bounded $\delta$-pair and $u':R\rightarrow R'$ is a ring homomorphism, is called a {\it relatively quasiregular semiperfectoid cover of $R$} if it satisfies the following conditions:
    \begin{itemize}
        \item[(1)] The $\delta$-pair $(A',R')$ is relatively quasiregular semiperfectoid, in the sense of Definition \ref{defn:relatively-quasiregular-semiperfectoids} (\cite[Definition 9.2]{AKN23}).

        \item[(2)] The map $u':R\rightarrow R'$ is $p$-completely faithfully flat. %and $L_{R'/R}\in \rmD(R')$ has $p$-complete Tor-amplitude concentrated in homological degree $[0,1]$, that is, $u:R\rightarrow R'$ is a $p$-quasisyntomic cover.

        \item[(3)] The ring $A'$ is $p$-torsion-free, and the cotangent complex $L\Omega^1_{R'/(R\widehat{\otimes}^L A')}\in \rmD(R')$ has $p$-complete Tor-amplitude concentrated in homological degrees $[0,1]$. 

    \end{itemize}
\end{defn}

\begin{rem}\label{rem:difference}
    A transversal object in the sense of Definition \ref{defn:transversal-objects} is obviously a relatively quasiregular semiperfectoid cover, but not conversely. Note that condition (3) in Definition \ref{defn:transversal-objects} for a prism $(A,I,u)\in R_\Prism$ may not be easy to verify. The notion of relatively quasiregular semiperfectoid covers will be more convenient in practice. 
\end{rem}

\begin{rem}\label{rem:homological-degree-1}
    Condition $(3)$ in Definition \ref{defn:relatively-quasiregular-semiperfectoid-covers} is equivalent to the following condition:
    \begin{itemize}
        \item[(3')] The ring $A'$ is $p$-torsion-free, and the cotangent complex $L\Omega^1_{R'/(R\widehat{\otimes}^L A')}\in \rmD(R')$ has $p$-complete Tor-amplitude concentrated in homological degree $1$.
    \end{itemize}
    To see that $(3)\Rightarrow \text{(3')}$, by definition of relatively quasiregular semiperfectoid rings we choose a factorization $A'\rightarrow A''\rightarrow R'$ with $A'\rightarrow A''$ $p$-completely relatively perfect and $A''\rightarrow R'$ surjective. Then we have a canonical equivalence
    $$
    L\Omega^1_{R'/(R\widehat{\otimes}^L A')}\otimes^L_{R'}R'/p\xrightarrow{\sim}
    L\Omega^1_{R'/(R\widehat{\otimes}^L A'')}\otimes^L_{R'}R'/p,
    $$
    and $L\Omega^1_{R'/(R\widehat{\otimes}^L A'')}$ is concentrated in homological degree $\geq 1$. Therefore $L\Omega^1_{R'/(R\widehat{\otimes}^L A')}$ must have $p$-complete Tor-amplitude concentrated in degree $1$. 
\end{rem}

As in Lemma \ref{lem:transveral-LOmegaR0} and Lemma \ref{lem:transversal-for-u-quasietale}, we have the following observations. 

\begin{lem}\label{lem:relative-qrsp-cover-LOmegaR0}
    Let $R$ be a $p$-quasisyntomic ring, and let $(A',R')$ be a relatively quasiregular semiperfectoid $\delta$-pair where $A'$ is $p$-torsion-free, and let $u':R\rightarrow R'$ be a $p$-completely faithfully flat ring homomorphism. If $L\Omega^1_R$ has $p$-complete Tor-amplitude concentrated in homological degree $[0,0]$, then $(A',R',u')$ is a relatively quasiregular semiperfectoid cover of $R$.
\end{lem}

\begin{proof}
    Consider the chain of ring homomorphisms $A'\rightarrow R\widehat{\otimes}^L A'\rightarrow R'$ and the associated cofiber sequence of cotangent complexes
    $$
    L\Omega^1_{(R\widehat{\otimes}^L A')/A'}\otimes^L_{R\widehat{\otimes}^LA'}R'\rightarrow L\Omega^1_{R'/A'}\rightarrow L\Omega^1_{R'/(R\widehat{\otimes}^LA')}, 
    $$
    and note that $L\Omega^1_{(R\otimes^L A')/A'}\otimes^L_{R\otimes^LA'}R'\simeq L\Omega^1_R\otimes^L_R R'$, and also note that $L\Omega^1_{R'/A'}$ has $p$-complete Tor-amplitude concentrated in $1$ by Remark \ref{rem:relatively-quasiregular-semiperfectoids}. Thus $L\Omega^1_{R'/(R\widehat{\otimes}^LA')}$ must have $p$-complete Tor-amplitude concentrated in $1$.
\end{proof}

\begin{lem}\label{lem:relatively-qrsp-cover-for-u'-quasietale}
    Let $R$ be a $p$-quasisyntomic ring, and let $(A',R')$ be a relatively quasiregular semiperfectoid $\delta$-pair where $A'$ is $p$-torsion-free, and let $u':R\rightarrow R'$ be a $p$-completely faithfully flat ring homomorphism. Then the following statements hold.
    \begin{itemize}
        \item[(1)] If $L\Omega^1_{R'/R}$ has $p$-complete Tor-amplitude concentrated in degree $1$ and $L\Omega^1_{A'}\otimes^L_{A'}R'\in \rmD(R')$ has $p$-complete Tor-amplitude concentrated in degree $[0,0]$, then $(A',R',u')$ is a relatively quasiregular semiperfectoid cover of $R$.

        \item[(2)] If $L\Omega^1_{R'/R}$ is $p$-completely zero, then $(A',R',u')$ is a relatively quasiregular semiperfectoid cover of $R$ if and only if $L\Omega^1_{A'}\otimes^L_{A'}R'\in \rmD(R')$ has $p$-complete Tor-amplitude concentrated in degree $[0,0]$.
    \end{itemize}
    Note that if $L\Omega^1_{A'}\in \rmD(A')$ has $p$-complete Tor-amplitude concentrated in degree $[0,0]$, that is, if $A'$ is quasismooth, then $L\Omega^1_{A'}\otimes^L_{A'}R'\in \rmD(R')$ has $p$-complete Tor-amplitude concentrated in degree $[0,0]$.
\end{lem}

\begin{proof}
    Consider the chain of ring homomorphisms $R\rightarrow R\widehat{\otimes}^L A'\rightarrow R'$ and the associated cofiber sequence of cotangent complexes 
    $$
     L\Omega^1_{(R\widehat{\otimes}^L A')/R}\otimes^L_{R\widehat{\otimes}^LA'}R'\rightarrow L\Omega^1_{R'/R}\rightarrow L\Omega^1_{R'/(R\widehat{\otimes}^LA')},
    $$
    and note that 
    $$
    L\Omega^1_{(R\otimes^L A')/R}\otimes^L_{R\otimes^L A'}R'\simeq L\Omega^1_{A'}\otimes^L_{A'}R'.
    $$
    Hence, the desired statements hold.
\end{proof}

\begin{ex}\label{ex:quasiregular-semiperfectoid-covers}
    Let $R$ be a $p$-quasisyntomic ring and let $R\rightarrow S$ be a $p$-quasisyntomic cover with $S$ a quasiregular semiperfectoid ring. Then the $\delta$-pair $(\Z_p,S)$, being relatively quasiregular semiperfectoid (see \cite[Example 9.4]{AKN23}), is a relatively quasiregular semiperfectoid cover of $R$. 
\end{ex}

\begin{ex}\label{ex:Zmodpn}
    Let $n\in \Z_{\geq 1}$ and consider $(\Z/p^n)_\Prism$, together with the $\delta$-pair 
    $$
    \Z[z]\xrightarrow{z\mapsto p+(p^n)}\Z/p^n
    $$ 
    where the $\delta$-structure on $\Z[z]$ is determined by $\delta(z)=0$. Then $(\Z[z], \Z/p^n)$ can be naturally regarded as a relatively quasiregular semiperfectoid cover of $\Z/p^n$: by Lemma \ref{lem:relatively-qrsp-cover-for-u'-quasietale}, we need only to check that $L\Omega^1_{\Z[z]}$ has $p$-complete Tor-amplitude concentrated in homological degree $[0,0]$, and this is immediate.
\end{ex}

\begin{construction}\label{construction:prism-of-relatively-qrsp}
    Let $R$ be a $p$-quasisyntomic ring, and let $(A',R',u')$ be a relatively quasiregular semiperfectoid cover of $R$. By Theorem \ref{thm:prismatic-cohomology-of-relative-quasiregular-semiperfectoids}, we have an initial object $(\Prism_{R'/A'},I\Prism_{R'/A'},\iota)\in (R'/A')_\Prism$ with $\iota:R'\rightarrow \Prism_{R'/A'}/I$ $p$-completely faithfully flat. Consider 
    $$
    \iota_R:R\xrightarrow{u'}R'\xrightarrow{\iota} \Prism_{R'/A'}/I.
    $$
    We obtain $(\Prism_{R'/A'},I\Prism_{R'/A'},\iota_R)\in R_\Prism$.
\end{construction}

We now explain why the object $(\Prism_{R'/A'},I\Prism_{R'/A'},\iota_R)\in R_\Prism$ covers the final object and admits finite self-coproducts in $R_\Prism$. In fact, coproducts with such an object exist in $R_\Prism$. %But we also have the following lifting technique. 
The idea is essentially the same as that of Construction \ref{construction:coproducts-in-RPrism}.

\begin{lem}\label{lem:coproducts-with-relatively-qrsp-covers}
    Let $R$ be a $p$-quasisyntomic ring, and let $(A',R',u')$ be a relatively quasiregular semiperfectoid cover of $R$. Then for any $(B,J,v)\in R_\Prism$ with $\overline{B}=B/J$, the $\delta$-pair $(A'\widehat{\otimes}^L B,R'\widehat{\otimes}^L_R\overline{B})$ is a relatively quasiregular semiperfectoid $\delta$-pair.
\end{lem}

\begin{proof}
    The ideal $A'\widehat{\otimes}^L JB$ exhibited $(A'\widehat{\otimes}^L B,R'\widehat{\otimes}^L_R\overline{B})$ as a pre-prismatic $\delta$-pair. It remains to check that $L\Omega^1_{(R'\otimes^L_R\overline{B})/(A'\otimes^L \overline{B})}$ has $p$-complete Tor-amplitude concentrated in homological degree $1$. Note that 
    $$
    L\Omega^1_{(R'\otimes^L_R\overline{B})/(A'\otimes^L \overline{B})}
    \simeq L\Omega^1_{R'/(A'\otimes^L R)}\otimes^L_R\overline{B}.
    $$
    It suffices to appeal to condition $(3')$ in Remark \ref{rem:homological-degree-1}.
\end{proof}

\begin{construction}\label{construction:coproducts-with-relatively-qrsp-covers}
    Let $R$ be a $p$-quasisyntomic ring, and let $(A',R',u')$ be a relatively quasiregular semiperfectoid cover of $R$, and let $(B,J,v)\in R_\Prism$. Set $\overline{B}=B/J$. Consider the $\delta$-pair $(A'\widehat{\otimes}^L B,R'\widehat{\otimes}^L_R\overline{B})$. By Lemma \ref{lem:coproducts-with-relatively-qrsp-covers}, $(A'\widehat{\otimes}^L B,R'\widehat{\otimes}^L_R\overline{B})$ is a relatively quasiregular semiperfectoid $\delta$-pair, and by Theorem \ref{thm:prismatic-cohomology-of-relative-quasiregular-semiperfectoids} (in fact, Theorem \ref{thm:prismatic-cohomology-of-prismatic-relative-quasiregular-semiperfectoids} is sufficient), we have an initial object 
    $$(\Prism_{(R'\widehat{\otimes}^L_R\overline{B})/(A'\widehat{\otimes}^L B)},J\Prism_{(R'\widehat{\otimes}^L_R\overline{B})/(A'\widehat{\otimes}^L B)},\iota)\in ((R'\widehat{\otimes}^L_R\overline{B})/(A'\widehat{\otimes}^L B))_\Prism,
    $$
    with $R'\widehat{\otimes}^L_R\overline{B}\rightarrow \Prism_{(R'\widehat{\otimes}^L_R\overline{B})/(A'\widehat{\otimes}^L B)}/J$ $p$-completely faithfully flat. We then obtain an object 
    $$(\Prism_{(R'\widehat{\otimes}^L_R\overline{B})/(A'\widehat{\otimes}^L B)},J\Prism_{(R'\widehat{\otimes}^L_R\overline{B})/(A'\widehat{\otimes}^L B)})\in R_\Prism.
    $$
    We claim that this object is a coproduct of $(\Prism_{R'/A'},I\Prism_{R'/A'})$ with $(B,J)$ in $R_\Prism$ (we safely ignore the structure maps from $R$), and the verification is routine and essentially the same as that in Construction \ref{construction:coproducts-in-RPrism}.
\end{construction}

\begin{prop}\label{prop:coproducts-with-relatively-qrsp-covers}
    Let $R$ be a $p$-quasisyntomic ring. Let $(A',R',u')$ be a relatively quasiregular semiperfectoid cover of $R$, and let $(B,J,v)\in R_\Prism$. Then there exists a coproduct $(\Prism_{R'/A'},I\Prism_{R'/A'})\coprod (B,J)$ in the category $R_\Prism$. Moreover, the canonical map $(B,J)\rightarrow (\Prism_{R'/A'},I\Prism_{R'/A'})\coprod (B,J)$ is $(p,J)$-completely faithfully flat.
\end{prop}

\begin{proof}
    This just summarizes Construction \ref{construction:coproducts-with-relatively-qrsp-covers}.
\end{proof}

\begin{cor}\label{cor:coproducts-with-relatively-qrsp-covers}
    Let $R$ be a $p$-quasisyntomic ring, and let $(A',R',u')$ be a relatively quasiregular semiperfectoid cover of $R$. Then $(\Prism_{R'/A'},I\Prism_{R'/A'},\iota_R)\in R_\Prism$ covers the final object (in the sense of Definition \ref{defn:cover-the-final-object}) and admits finite self-coproducts in $R_\Prism$.
\end{cor}

\begin{proof}
    This follows immediately from Proposition \ref{prop:coproducts-with-relatively-qrsp-covers}.
\end{proof}

\section{Some Examples}\label{section:some-examples}

In this section, we include some examples which constantly appear in everyday practice of prismatic cohomology. For a $p$-quasisyntomic ring $R$ and $(A,I,u)\in R_\Prism$ which admits finite self-coproducts, we let $(A^{n},IA^n,u^n)$ ($n\geq0$) denote the $n+1$-fold self-coproduct of $(A,I,u)$ in $R_\Prism$, and let $(A^\bullet,IA^\bullet,u^\bullet)$ denote the corresponding cosimplicial object $\Delta\rightarrow R_\Prism$ given by $[n]\mapsto (A^n,IA^n,u^n)$. It is easy to see that this cosimplicial object is actually a cogroupoid object. 

\begin{ex}[Breuil-Kisin Prisms, {\cite[Example 2.8]{BS23}}]\label{ex:Breuil-Kisin-cover}
    Let $K$ be a finite extension of $\Q_p$, and let $\calO_K$ be the ring of integers of $K$. Choose a uniformizer $\varpi\in \calO_K$, and let $k=\calO_K/\varpi$ be the residue field. Finally, let $E_K(z)\in W(k)[z]$ be a minimal polynomial of $\varpi$. By Example \ref{ex:Breuil-Kisin} and Corollary \ref{cor:coproducts-in-RPrism}, the Breuil-Kisin prism $(\frakS=W(k)[[z]],(E_K(z)))\in (\calO_K)_\Prism$ covers the final object and admits finite self-coproducts. By Lemma \ref{lem:cover-the-final-object}, we have the following equivalences of categories
    $$
    \hatD_\crys((\calO_K)_\Prism,\calO_\Prism)\xrightarrow{\sim}\varprojlim_{\bullet\in \Delta}\hatD(\frakS^\bullet),
    $$
    and 
    $$
    \Vect((\calO_K)_\Prism,\calO_\Prism)\xrightarrow{\sim}\varprojlim_{\bullet\in \Delta}\Vect(\frakS^\bullet).
    $$
    Thus the prismatic crystals on $(\calO_K)_\Prism$ can be studied as $(p,E_K)$-complete $(\frakS^0=\frakS,\frakS^1)$-comodules, where $(\frakS^0,\frakS^1)$ has a natural Hopf algebroid structure. The object $\frakS^1$ can be described more precisely: by Construction \ref{construction:coproducts-in-RPrism}, we identify $\frakS^1$ with $\Prism_{\calO_K/W(k)[z_0,z_1]}$, and by Theorem \ref{thm:prismatic-cohomology-of-relative-quasiregular-semiperfectoids}, we have 
    $$
    \frakS^1=\Prism_{\calO_K/W(k)[z_0,z_1]}=W(k)[z_0,z_1]\{\frac{z_1-z_0}{E_K(z_0)}\}^\wedge_{(p,E_K(z_0))}.
    $$
\end{ex}

\begin{ex}[$q$-de Rham Prisms]\label{ex:q-de-Rham-cover}
    Assume that $p$ is an odd prime. Let $\zeta_p$ be a primitive $p$-th root of unity, and consider $\Z_p[\zeta_p]$. By an argument similar to that of Example \ref{ex:Breuil-Kisin}, we see that the $q$-de Rham prism $(\qdR=\Z_p[[q-1]],([p]_q))\in \Z_p[\zeta_p]_\Prism$ is a transversal cover and therefore covers the final object and admits finite self-coproducts. Thus the prismatic crystals on $\Z_p[\zeta_p]_\Prism$ can be studied as $(p,[p]_q)$-complete comodules over the Hopf algebroid $(\qdR^0,\qdR^1)$, where 
    $$
    \qdR^1=\Prism_{\Z_p[\zeta_p]/\Z_p[q_0,q_1]}=\Z_p[q_0,q_1]\{\frac{q_1-q_0}{[p]_{q_0}}\}^\wedge_{(p,[p]_{q_0})}.
    $$
\end{ex}

\begin{ex}[$\tilp$-de Rham Prisms]\label{ex:p-de-Rham-cover}
    Assume that $p$ is an odd prime. By an argument similar to that of Example \ref{ex:Breuil-Kisin}, we see that the $\tilp$-de Rham prism (see \cite[Notation 3.8.9]{BL22}) $(\pdR=\Z_p[[\tilp]],(\tilp))\in (\Z_p)_\Prism$ is a transversal cover and therefore covers the final object and admits finite self-coproducts. Thus the prismatic crystals on $(\Z_p)_\Prism$ can be studied as $(p,\tilp)$-complete comodules over the Hopf algebroid $(\pdR^0,\pdR^1)$, where 
    $$
    \pdR^1=\Prism_{\Z_p/\Z_p[\tilp_0,\tilp_1]}=\Z_p[\tilp_0,\tilp_1]\{\frac{\tilp_1}{\tilp_0}\}^\wedge_{(p,\tilp_0)}.
    $$
\end{ex}

\begin{ex}\label{ex:qdR-polynomial-cover}
    Assume that $p$ is an odd prime. Let $\zeta_p$ be a primitive $p$-th root of unity and consider $\Z_p[\zeta_p][x_1,\ldots,x_n]$, the polynomial ring in the indeterminates $x_1,\ldots,x_n$ with coefficients in $\Z_p[\zeta_p]$. Consider the $\delta$-ring defined as follows:
    $$
    \qdR_n:=\Z_p[[q-1]][x_1,\ldots,x_n]^\wedge_{(p,[p]_q)},
    $$
    whose $\delta$-ring structure is determined by $\delta(x_i)=0$, $1\leq i\leq n$. Then by taking $\Z_p[\zeta_p][x_1,\ldots,x_n]\xrightarrow{x_i\mapsto x_i}\qdR_n/([p]_q)$ we obtain a natural object 
    $$
    (\qdR_n,([p]_q))\in \Z_p[\zeta_p][x_1,\ldots,x_n]_\Prism,
    $$
    which is a transversal cover by an argument similar to that of Example \ref{ex:Breuil-Kisin}. Thus the prismatic crystals on $\Z_p[\zeta_p][x_1,\ldots,x_n]_\Prism$ can be studied as $(p,[p]_q)$-complete comodules over the Hopf algebroid $(\qdR_n^0,\qdR_n^1)$, where 
    \begin{align*}
    \qdR_n^1
    & =\Prism_{\Z_p[\zeta_p][x_1,\ldots,x_n]/\Z_p[q,q'][x_1,\ldots,x_n,x_1',\ldots,x_n']}\\
    & =\Z_p[q,q'][x_1,\ldots,x_n,x_1',\ldots,x_n']\{\frac{q'-q}{[p]_{q}},\frac{x'_1-x_1}{[p]_{q}},\ldots,\frac{x_n'-x_n}{[p]_{q}}\}^\wedge_{(p,[p]_{q})}.
    \end{align*}
\end{ex}

With the notion of relatively quasiregular semiperfectoid covers, we can consider slightly more complicated objects.

\begin{ex}\label{ex:logarithmic-Breuil-Kisin-cover}
    Let $K$ be a finite extension of $\Q_p$, and let $\calO_K$ be the ring of integers of $K$. Choose a uniformizer $\varpi\in \calO_K$, and let $k=\calO_K/\varpi$ be the residue field. Finally, let $E_K(z)\in W(k)[z]$ be a minimal polynomial of $\varpi$. Consider the $\delta$-pair 
    $$
    W(k)[z,v]\xrightarrow{z\mapsto \varpi,v\mapsto 1}\calO_K,
    $$
    where the $\delta$-ring structure on $W(k)[z,v]$ is determined by $\delta(z)=\delta(v)=0$. Then by Lemma \ref{lem:relatively-qrsp-cover-for-u'-quasietale}, the $\delta$-pair $(W(k)[z,v],\calO_K)$ can be naturally regarded as a relatively quasiregular semiperfectoid cover of $\calO_K$. Therefore, by Corollary \ref{cor:coproducts-with-relatively-qrsp-covers}, the object 
    $$
    (\frakS_{\log}:=\Prism_{\calO_K/W(k)[z,v]},E_K(z)\frakS_{\log},\iota_{\calO_K})\in (\calO_K)_\Prism
    $$
    covers the final object and admits finite self-coproducts, where by Theorem \ref{thm:prismatic-cohomology-of-relative-quasiregular-semiperfectoids},  
    $$
    \frakS_{\log}=\Prism_{\calO_K/W(k)[z,v]}=W(k)[z,v]\{\frac{1-v}{E_k(z)}\}^\wedge_{(p,E_K(z))}.
    $$
    Thus the prismatic crystals on $(\calO_K)_\Prism$ can be studied as $(p,E_K(z))$-complete comodules over the Hopf algebroid $(\frakS_{\log}^0,\frakS_{\log}^1)$, where 
    $$
    \frakS_{\log}^1=\Prism_{\calO_K/W(k)[z_0,z_1,v_0,v_1]}=W(k)[z_0,z_1,v_0,v_1]\{\frac{z_1-z_0}{E_K(z_0)},\frac{v_1-v_0}{E_K(z_0)},\frac{v_0-1}{E_K(z_0)}\}^\wedge_{(p,E_K(z_0))}.
    $$

    Note that there is a natural morphism $(\frakS^0,\frakS^1)\rightarrow (\frakS_{\log}^0,\frakS_{\log}^1)$, where $(\frakS^0,\frakS^1)$ is introduced in Example \ref{ex:Breuil-Kisin-cover}. This morphism is in fact an equivalence of Hopf algebroids, in the sense that it induces an isomorphism of the corresponding formal stacks.
\end{ex}

\begin{ex}\label{ex:complete-intersection-Zpzetap}
    Assume that $p$ is an odd prime. Let $\zeta_p$ be a primitive $p$-th root of unity and consider $R=(\Z_p[\zeta_p][x_1,\ldots,x_n]/(f_1,\ldots,f_r))^\wedge_p$, where $f_1,\ldots,f_r$ form a regular sequence in $\Z_p[\zeta_p][x_1,\ldots,x_n]$. Consider the $\delta$-pair
    $$
    \Z_p[q][x_1,\ldots,x_n]\xrightarrow{q\mapsto \zeta_p,x_i\mapsto x_i}R, 
    $$
    where the $\delta$-ring structure on $\Z_p[q][x_1,\ldots,x_n]$ is determined by $\delta(q)=\delta(x_i)=0$ for $1\leq i\leq n$. Then, by Lemma \ref{lem:relatively-qrsp-cover-for-u'-quasietale}, the $\delta$-pair $(\Z_p[q][x_1,\ldots,x_n],R)$ can be naturally regarded as a relatively quasiregular semiperfectoid cover of $R$. Therefore, by Corollary \ref{cor:coproducts-with-relatively-qrsp-covers}, the object 
    $$
    (\qdR_{n,f_1,\ldots,f_r}:=\Prism_{R/\Z_p[q][x_1,\ldots,x_n]},[p]_q\qdR_{n,f_1,\ldots,f_r},\iota_R)\in R_\Prism
    $$
    covers the final object and admits finite self-coproducts, where, by Theorem \ref{thm:prismatic-cohomology-of-relative-quasiregular-semiperfectoids}, 
    $$
    \qdR_{n,f_1,\ldots,f_r}=\Z_p[q,x_1,\ldots,x_n]\{\frac{f_1}{[p]_q},\ldots,\frac{f_r}{[p]_q}\}^\wedge_{(p,[p]_q)}.
    $$
    Thus the prismatic crystals on $R_\Prism$ can be studied as $(p,[p]_q)$-complete comodules over the Hopf algebroid $(\qdR_{n,f_1,\ldots,f_r}^0,\qdR_{n,f_1,\ldots,f_r}^1)$, where $\qdR_{n,f_1,\ldots,f_r}^1=\Prism_{R/\Z_p[q,q',x_1,\ldots,x_n,x_1',\ldots,x_n']}$ can be expressed as 
    $$
    \Z_p[q,q',x_1,\ldots,x_n,x_1',\ldots,x_n']\{\frac{q'-q}{[p]_q},\frac{x_i'-x_i}{[p]_q},\frac{f_j}{[p]_q}\mid 1\leq i\leq n,1\leq j\leq r\}^\wedge_{(p,[p]_q)}.
    $$
\end{ex}

\begin{ex}\label{ex:a-cover-for-divided-power-polynomial}
    Let $R=(\Z_p[x]_{\pd})^\wedge_p$, where $\Z_p[x]_\pd$ denotes the $p$-divided polynomial ring in the indeterminate $x$, that is, 
    $$
    \Z_p[x]_\pd=\Z_p[x_0=x,x_1,\ldots,x_n,\ldots]/(px_i-x_{i-1}^p\mid i\geq 1).
    $$
    Consider the $\delta$-pair
    $$
    \Z_p[z][x_0,x_1,\ldots]\xrightarrow{z\mapsto p,x_i\mapsto x_i} R,
    $$
    where the $\delta$-ring structure on $\Z_p[z][x_0,\ldots,x_n]$ is determined by $\delta(z)=\delta(x_i)=0$ for all $i\geq 0$. Then, by Lemma \ref{lem:relatively-qrsp-cover-for-u'-quasietale}, the $\delta$-pair $(\Z_p[z][x_0,x_1,\ldots],R)$ can be naturally regarded as a relatively quasiregular semiperfectoid cover of $R$. Therefore, by Corollary \ref{cor:coproducts-in-RPrism}, the object 
    $$
    (\frakS_{[x]_\pd}:=\Prism_{R/\Z_p[z][x_0,x_1,\ldots]},(z-p)\frakS_{[x]_\pd},\iota_R)\in R_\Prism
    $$
    covers the final object and admits finite self-coproducts, where, by Theorem \ref{thm:prismatic-cohomology-of-relative-quasiregular-semiperfectoids}, 
    $$
    \frakS_{[x]_\pd}=\Z_p[z,x_0,x_1,\ldots]\{\frac{px_i-x_{i-1}^p}{z-p}\mid i\geq 1\}^\wedge_{(p,z-p)}.
    $$
    Thus the prismatic crystals on $R_\Prism$ can be studied as $(p,z-p)$-complete comodules over the Hopf algebroid $(\frakS_{[x]_\pd}^0,\frakS_{[x]_\pd}^1)$, where 
    $$
    \frakS_{[x]_\pd}^1=\Z_p[z,z',x_i,x_i'\mid i\geq 0]\{\frac{z'-z}{z-p},\frac{x_i'-x_i}{z-p},\frac{px_j-x_{j-1}^p}{z-p}\mid i\geq 0,j\geq 1\}^\wedge_{(p,z-p)}.
    $$
\end{ex}

\end{document}